\documentclass{article}
\usepackage{amssymb,amsmath,amsfonts,amsthm,latexsym}
\usepackage{graphicx}

\newtheorem{theo}{Theorem}[section]
\newtheorem{coro}[theo]{Corollary}
\newtheorem{prop}[theo]{Proposition}
\newtheorem{lemm}[theo]{Lemma}

\theoremstyle{definition}
\newtheorem{rema}[theo]{Remark}
\newtheorem{defi}[theo]{Definition}

 \def\RR{{\mathbb R}}  \def\TT{{\mathbb T}}
   
 \def\ZZ{{\mathbb Z}}

\def\De{\Delta}

\def\Ga{\Gamma}

  \def\cG{{\cal G}}  
    
    \def\cU{{\cal U}}
    
    \def\cW{{\cal W}}
\def\cF{{\cal F}}

\title{Anomalous partially hyperbolic diffeomorphisms I: dynamically coherent examples}

\author{Christian Bonatti, Kamlesh Parwani\thanks{This work was partially supported by a grant from the Simons Foundation (\#280151 to Kamlesh Parwani).}, and  Rafael Potrie\thanks{R.~Potrie was partially supported by FCE-2011-6749, CSIC grupo 618 and the Palis-Balzan project.}}

\begin{document}

\maketitle

\begin{abstract}
We build an example of a non-transitive, dynamically coherent
partially hyperbolic diffeomorphism $f$ on a closed $3$-manifold
with exponential growth in its fundamental group such that $f^n$
is not isotopic to the identity for all $n\neq 0$.  This example
contradicts a conjecture in \cite{HHU}. The main idea is to
consider a well-understood time-$t$ map of a non-transitive Anosov
flow and then carefully compose with a Dehn twist.

\medskip

{\bf Diff\'eomorphismes partiellement hyperboliques anormaux I:
exemples dynamiquement coh\'erents}

\smallskip

R\'esum\'e: Sur une $3$-vari\'et\'e ferm\'ee dont le groupe
fondamental est \`a croissance exponentielle,  nous construisons
un exemple de diff\'eomorphisme $f$, partiellement hyperbolique,
dynamiquement coh\'erent, non-transitif, et dont aucune puissance
$f^n$, $n\neq 0$, n'est isotope \`a l'identit\'e.  Cet exemple
infirme une conjecture de \cite{HHU}. L'exemple est obtenu en
composant avec soin le temps $t$ d'un flot d'Anosov non-transitif
bien choisi avec un twist de Dehn.

{ \medskip \noindent \textbf{Keywords:} Partially hyperbolic
diffeomorphisms, classification.

  }

\end{abstract}

\section{Introduction}
In recent years, partially hyperbolic diffeomorphisms have been
the focus of considerable study. Informally, partially hyperbolic
diffeomorphisms are generalization of hyperbolic maps. The
simplest partially hyperbolic diffeomorphisms admit an invariant
splitting into three bundles: one of which is uniformly contracted
by the derivative, another which is uniformly expanded, and a
center direction whose behavior is intermediate. A more formal
definition will soon follow.  In this paper, we shall restrict to
the case of these diffeomorphisms on closed 3-dimensional
manifolds.


The study of partially hyperbolic diffeomorphisms has followed two
main directions. One consist of studying conditions under which a
volume preserving partially hyperbolic diffeomorphism is stably
ergodic. This is not the focus of this article. See \cite{HHU3, W}
for recent surveys on this subject.

The second direction, initiated in \cite{BW,BBI}, has as a long
term goal of classifying these partially hyperbolic systems, at
least topologically. Even in dimension 3, this goal seems quite
ambitious but some  partial progress has been made, which we
briefly review below. This paper is intended to further this
classification effort by providing new examples of partially
hyperbolic diffeomorphisms. In a forthcoming paper (\cite{BP}) we
shall provide new transitive examples (and even stably ergodic);
their construction uses some of the ideas of this paper as well as
some new ones. Another viewpoint might be that these new examples
throw a monkey wrench into the classification program. In light of
these new partially hyperbolic diffeomorphisms, is there any hope
to achieve any reasonable sense of a classification?

\subsection{Preliminaries}
Before diving into a detailed exposition of these examples, we
provide the necessary definitions and background. Let $M$ be a
closed $3$-manifold, we say that a diffeomorphism $f: M \to M$ is
\emph{partially hyperbolic} if the tangent bundle splits into
three one-dimensional\footnote{One of the advantages of working
with one-dimensional bundles is that the norm of $Df$ along such
bundles controls the contraction/expantion of every vector in the
bundle. Compare with definitions of partial hyperbolicity when the
bundles are not one-dimensional \cite{W}.} $Df$-invariant
continuous bundles $TM= E^{ss}\oplus E^c \oplus E^{uu}$ such that
there exists $\ell
>0$ such that for every $x \in M$:

$$ \|Df^\ell|_{E^{ss}(x)} \| < \min \{1, \|Df^\ell|_{E^c(x)}\|\} \leq
\max \{1, \|Df^\ell|_{E^c(x)} \} < \|Df^\ell|_{E^{uu}(x)} \|. $$

Sometimes, the more restrictive notion of \emph{absolute partial
hyperbolicity} is used. This means that $f$ is partially
hyperbolic and there exists $\lambda < 1 < \mu$ such that:

$$ \|Df^\ell|_{E^{ss}(x)} \| < \lambda<  \|Df^\ell|_{E^c(x)}\| < \mu < \|Df^\ell|_{E^{uu}(x)}
\|. $$

For the classification of such systems, one of the main obstacles
is understanding the existence of invariant foliations tangent to
the center direction $E^c$. In general, the bundles appearing in
the invariant splitting are not regular enough to guaranty unique
integrability. In the case of the strong stable $E^{ss}$ and
strong unstable $E^{uu}$ bundles, dynamical arguments insure the
existence of unique foliations tangent to the strong stable and
unstable bundle (see for example \cite{HPS}). However, the other
distributions need not be integrable.

The diffeomorphism $f$ is \emph{dynamically coherent} if there are
$2$-dimensional $f$-invariant foliations $\cW^{cs}$ and $\cW^{cu}$
tangent to the distributions $E^{ss}\oplus E^c$ and $E^c\oplus
E^{uu}$, respectively. These foliations, when they exist,
intersect along a $1$-dimensional foliation $\cW^{c}$ tangent to
$E^c$. The diffeomorphism $f$ is \textit{robustly dynamically
coherent} if there exists a $C^1$-neighborhood of $f$ comprised
only of dynamically coherent partially hyperbolic diffeomorphisms.
There is an example of non dynamically coherent partially
hyperbolic diffeomorphisms\footnote{We should remark that this
example is not absolutely partially hyperbolic. Moreover, it is
isotopic to one of the known models of partially hyperbolic
diffeomorphisms.} on $\TT^3$ (see \cite{HHU1}). This example is
not transitive and it is not known whether every transitive
partially hyperbolic diffeomorphisms on a compact $3$ manifold is
dynamically coherent. See \cite{HP} and references therein for the
known results on dynamical coherence of partially hyperbolic
diffeomorphisms in dimension 3.

Two dynamically coherent partially hyperbolic diffeomorphisms
$f\colon M\to M$ and $g\colon N\to N$ are \emph{leaf conjugate} if
there is a homeomorphism $h\colon M\to N$ so that $h$ maps the
center foliation of $f$ on the center foliation of $g$ and  for
any $x\in M$  the points $h(f(x))$ and $g(h(x))$ belong to the
same center leaf of $g$.

Up to now the only known example of dynamically coherent partially
hyperbolic diffeomorphisms were, up to finite lift and finite
iterates, leaf conjugate to one of the following models:

\begin{enumerate}
 \item linear Anosov automorphism of $\TT^3$;
 \item skew products over a linear Anosov map of the torus $\TT^2$;
 \item time one map of an Anosov flow.
\end{enumerate}

It has been conjectured, first in the transitive case (informally
by Pujals in a talk and then written in \cite{BW}), and later in
the dynamically coherent case (in many talks and minicourses
\cite{HHU}) that every partially hyperbolic diffeomorphism should
be, up to finite cover and iterate, leaf conjugate to one of these
three models. Positive results have been obtained in
\cite{BW,BBI,HP} and some families of 3-manifolds are now known
only to admit partially hyperbolic diffeomorphisms that are leaf
conjugate to the models.

\subsection{Statements of results}
The aim of this paper is to provide a counter example to the
conjecture stated above. Our examples are not isotopic to any of
the models.

In order to present the ideas in the simplest way, we have chosen
to detail the construction of a specific example on a (possibly
the simplest) manifold admitting a non-transitive Anosov flow
transverse to a non-homologically trivial incompressible
two-torus; the interested reader should consult \cite{Br} for more
on 3-manifolds admitting such non-transitive Anosov flows. Our
arguments go through directly in some other manifolds, but for
treating the general case of 3-manifolds admitting non-transitive
Anosov flows further work must be done.

\begin{theo}\label{t.main} There is a  closed orientable $3$-manifold $M$ endowed with a non-transitive Anosov flow $X$
and a diffeomorphism $f\colon M\to M$ such that:
\begin{itemize}
\item $f$ is absolutely partially hyperbolic; \item $f$ is
robustly dynamically coherent;
 \item the restriction of $f$ to its chain recurrent set coincides with the time-one map of the Anosov flow $X$, and
 \item for any $n\neq 0$, $f^n$ is not isotopic to the identity.
\end{itemize}
\end{theo}

The manifold $M$ on which our example is constructed also admits a
transitive Anosov flow (see \cite{BBY}).

As a corollary of our main theorem, we show that $f$ is a counter
example to the conjecture stated above in the non-transitive case
(see \cite{HHU,HP}):

\begin{coro}\label{c.conjecture}
Let $f$ be the diffeomorphism announced in Theorem~\ref{t.main}.
Then for all $n$ the diffeomorphism $f^n$ does not admit a finite
lift that is leaf conjugate to any of the following:
\begin{itemize}
\item linear Anosov diffeomorphisms on $\TT^3$;
 \item partially hyperbolic skew product with circle fiber over an Anosov
diffeomorphism on the torus $\TT^2$; \item the time-one map of an
Anosov flow.
\end{itemize}
\end{coro}

\subsection{Organization of the paper}
The paper is organized in the following manner. In Section
\ref{s.DA} we describe a modified \textit{DA} diffeomorphism of
$\TT^2$. This particular \textit{DA} diffeomorphism of the torus
may not seem the simplest but it has the necessary properties that
make our example easy to present using only elementary methods. In
section \ref{s.FW}, we detail the construction of a non-transitive
Anosov flow, following the construction of Franks and Williams in
\cite{FW}. In Section \ref{s.perturbation} we establish
coordinates in a model space in order to prepare for the
appropriate perturbation diffeomorphism---a Dehn twist along a
separating torus $T_1$. Then, in Section \ref{s.N}, we choose the
length of the neighborhood of the separating torus $T_1$. We
present the example in Section \ref{s.example} after providing
criteria for establishing partial hyperbolicity in Section
\ref{s.Criteria}. Then, in Sections \ref{s.coherence} and
\ref{s.isotopy} we show that the example is dynamically coherent
and not leaf conjugate to previously known examples; it is also
not isotopic to the identity.  Next, in Section 2, we informally
outline the construction of our specific example.

\section{Informal presentation of the example}

The example is constructed in the following manner. We begin with
a $DA$ map with two sources instead of one.  This choice makes it
possible to easily show that our partially hyperbolic
diffeomorphism has no non-trivial iterate isotopic to the
identity.  Next, we build a non-transitive Anosov flow (apr\`es
\cite{FW}) transverse to a torus $T_1$ (this is always the case
for non-transitive Anosov flows \cite{Br}). Then our example is
obtained by composing the time $N$-map of this Anosov flow with a
Dehn twist along a neighborhood of the torus $T_1$ of the form

$$\bigcup_{t \in [0,N]} X_t (T_1),$$

\noindent which is diffeomorphic to $[0,1] \times \TT^2$.

The neighborhood and the time $N$ is chosen in order to preserve
partial hyperbolicity. In a nutshell, the idea is that a small
$C^1$-perturbation always preserves partial hyperbolicity, and so,
if we make the perturbation in a long enough neighborhood of
$T_1$, by insuring that the time interval  $[0,N]$ is sufficiently
large, the effect of the Dehn twist can be made to appear
negligible at the level of the derivative, even though the
$C^0$-distance cannot be made arbitrarily small\footnote{Note that
these statements are not precise and the remarks in this section
are intended to impart our intuition to the reader.}. More
precisely, we obtain conditions under which transversality between
certain bundles are preserved under this kind of composition which
allows to show partial hyperbolicity.

Since the perturbation is made in the wandering region of the
time-$N$ of the Anosov flow, the properties of the chain-recurrent
set are preserved and it is possible to study the integrability of
the center bundle by simply defining it in the obvious way and
showing that it plays well with attracting and repelling regions
as it approaches them. The fact that center leaves cannot be fixed
when they pass through the fundamental region where the
perturbation is made becomes a matter of checking that the new
intersections cannot be preserved by the altered dynamics.

There are some reasons for which we present a specific example:

\begin{enumerate}

\item Even though most of our arguments are quite general and our
Dehn twist perturbation can be applied to infinite family of
Franks-Williams type Anosov flows, it is easier to first see these
ideas presented for a single example. The construction of the
perturbation and the fact that it preserves partial hyperbolicity
is much easier to check in a specific case and we believe that
this narrative makes the global argument more transparent.

\item It is possible to apply these techniques to other types of
examples at the expense of having to check a few minor details.
However, to perform the examples in any manifold admitting a
non-transitive Anosov flow, some more work is required to
guarantee the transversality of the foliations after perturbation.
We believe this is beyond the scope of this paper and relegated to a more detailed study in a
forthcoming article. 

\item Mainly, for the specific example it is quite easy to give a
direct and intuitive argument to show that the resulting dynamics
has no iterate isotopic to the identity. This is carried out in
Section~\ref{s.isotopy} where we use the fact that the torus $T_1$
is homologically non-trivial and just utilize elementary algebraic
topology to show that the action on homology is non-trivial. For
the general case, showing that the perturbation is not isotopic to
the identity requires more involved arguments. We remark that the
paper \cite{McC} solves this problem in many situations, but in a
less elementary manner. Again, these details are best left to be expounded in another article.

\end{enumerate}

\section{Modified $DA$ map on $\TT^2$}\label{s.DA}

Our construction begins with a diffeomorphism of the torus. We use
a modified \textit{DA} map, with two sources instead of one.

We simply state the properties of the required map below.  The
classical construction of the Derived from Anosov (\textit{DA})
diffeomorphism is well known---see \cite{Ro} for instance. Our
modified \textit{DA} map is obtained by lifting (some iterate of)
the classical $DA$ map  to some $2$-folded cover of $\TT^2$.
Alternatively, one may begin with a linear Anosov diffeomorphism
with 2 fixed points and then create two sources by ``blowing up''
these fixed points and their unstable manifolds.

\begin{prop}\label{t.DA} There exists a diffeomorphism $\varphi\colon \TT^2\to \TT^2$ with the following properties:
\begin{itemize}
  \item the non-wandering set of $\varphi$ consists in one non-trivial
 hyperbolic attractor $A$ and two fixed sources $\sigma_1,\sigma_2$;
 \item the stable foliation of the hyperbolic attractor coincides with a linear
 (for the affine structure on $\TT^2=\RR^2/\ZZ^2$) irrational foliation on
 $\TT^2\setminus\{\sigma_1,\sigma_2\}$.
\end{itemize}
\end{prop}

\section{Building non-transitive Anosov flows}\label{s.FW}

In this section, we briskly run through the relevant details of
the classical Franks-Williams construction in (\cite{FW}) of a
non-transitive Anosov flow. Note that the map $\phi$ in \cite{FW}
is the standard $DA$ map with a single source, and so, the
presentation below has been adapted to our context.

Let $(M_0, Z)$ be the suspension of the $DA$-diffeomorphism
$\varphi$ given in Theorem~\ref{t.DA}. We denote by $\gamma_i$ the
periodic orbit of the flow of $Z$ corresponding to the sources
$\sigma_i$, and by $A_Z$ the hyperbolic attractor of the flow of
$Z$.

\begin{lemm}\label{l.coordinates} There is a convex map $\alpha\colon (0,\frac12)\to \RR$ tending to $+\infty$ at
 $0$ and $\frac 12$, whose derivative vanishes exactly at $\frac 14$, and  so that, for any
 $i\in1,2$ there is a tubular neighborhood $\Ga_i$ of $\gamma_i$ whose boundary is an embedded torus
$T_i\simeq \TT^2$ so that
\begin{itemize}
\item $T_i$ is transverse to $Z$, and therefore to the weak stable
foliation $W^{cs}_Z$ of the attractor $A_Z$;  We denote by $F^s_i$
the $1$-dimensional foliation induced by $W^{cs}_Z$ on $T_i$;
 \item there are coordinates $\theta_i\colon \TT^2\to T_i$ so that, in these coordinates:
 \begin{itemize}
 \item $F^s_i$ has exactly $2$ compact leaves $\{0\}\times S^1$ and $\{\frac 12\}\times S^1$;
 \item given  any leaf $L$ of $F^s_i$ in $(0,\frac 12)\times S^1$ there is $t\in\RR$
 so that  $L$ is the projection on $\TT^2$ of the graph of $\alpha(x)+t$;

 \item given  any leaf $L$ of $F^s_i$ in $(\frac 12,1)\times S^1$ there is $t\in\RR$
 so that  $L$ is the projection on $T^2$ of the graph of $\alpha(x-\frac 12)+t$.
 \end{itemize}
\end{itemize}
\end{lemm}

\begin{proof} See page 165 of \cite{FW}.
\end{proof}

Notice that the expression of $F^s_1$ and $F^s_2$ are the same in
the chosen coordinates. We denote by $F^u_i$ the image of $F^s_i$
by the translation $(x,y)\mapsto (x+\frac 14,y)$.
\begin{coro} $F^s_i$ and $F^u_j$ are transverse.
Moreover $F^s_i$ and $F^u_i$ are invariant under any translation
in the second coordinates (that is $(x,y)\mapsto (x,y+t)$,
$t\in\RR$).
\end{coro}

\begin{figure}[ht]\begin{center}
\input{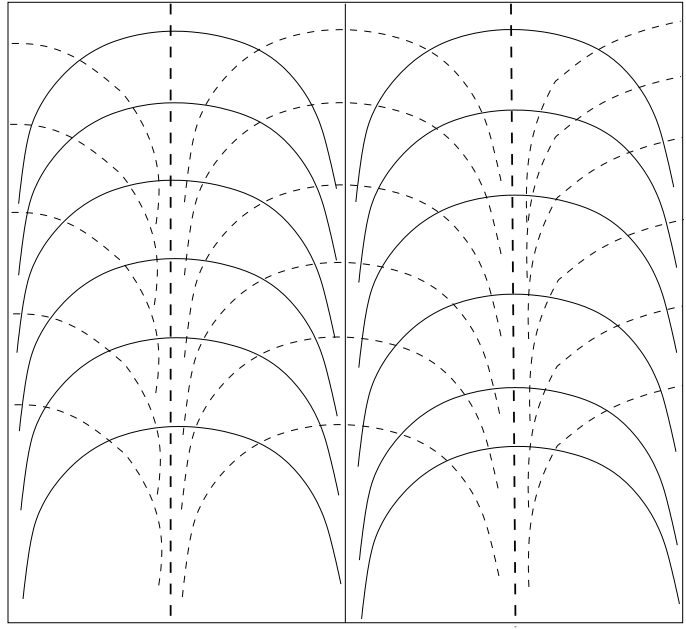}
\caption{\small{Gluing of the two foliations making them
transverse. Notice that both foliations are invariant under
vertical translations. }}\label{figure2}
\end{center}\end{figure}

We are now ready to define our manifold $M$ and the vector-field
$X$ on $M$.

Let $M^+$ be the manifold with boundary obtained by removing to
$M_0$ the interior of $\Ga_1\cup \Ga_2$, and we denote by $Z^+$ he
restriction of $Z$ to $M^+$.  We denote by $T_i^+$ the boundary
component of $M^+$ corresponding to the boundary of $\Ga_i$.
Notice that $T_i^+$ are tori transverse to $Z^+$ and $Z^+$ points
inwards to $M^+$.

Let $M^-$ be another copy of $M^+$ and we denote by $Z^-$ the
restriction of $-Z$ to $M^-$. We denote by $T_i^-$ the boundary
component of $M^-$, there are transverse to $Z^-$ and $Z^-$ is
points outwards of $M^-$.

We denote by $\psi\colon \partial M^+\to \partial M^-$ the
diffeomorphisms sending $T_i^+$ on $T_i^-$, $i=1,2$, and whose
expression in the $\TT^2$ coordinates are $(x,y)\mapsto
(x+\frac14,y)$.

We denote by $M$ the manifold obtained by gluing $M^+$ with $M^-$
along the diffeomorphism $\psi$. We can now restate
Franks-Williams theorem in our setting:

\begin{theo} There is a smooth structure on $M$, coinciding with the smooth structures on
$M^+$ and on $M^-$ and there is a smooth vector-field $X$ on $M$,
whose restrictions to $M^+$ and $M^-$ are  $Z^+$ and $Z^-$
respectively.

Furthermore $X$ is an Anosov vector-field whose non-wandering set
consists exactly in a non-trivial hyperbolic attractor $A_X$
contained in $M^+$ and a non-trivial hyperbolic repeller $R_X$
contained in $M^-$.
\end{theo}

A non-singular vector field $X$ on a $3$-manifold is an
\emph{Anosov vector field} if the time-$t$ map $X_t$ is partially
hyperbolic for some $t >0$ (see \cite{Br,FW} for the classical
definition). It is easy to show that the time $t$-map of an Anosov
flow will always be absolutely partially hyperbolic.

\begin{rema}\label{r.coherenceFlows}
Since the bundle generated by $X$ must be necessarily the center
bundle of $X_t$ it follows that the center direction is always
integrable for Anosov flows. Indeed, it is a classical result (see
e.g. \cite{HPS}) that Anosov flows are dynamically
coherent\footnote{This is not the way it is usually stated.
Dynamical coherence is modern terminology.}. We call $W^{cs}_X$
and $W^{cu}_X$ the center stable and center unstable foliations
(sometimes called weak stable and weak unstable foliations in the
context of Anosov flows) and $W^{ss}_X$ and $W^{uu}_X$ the strong
stable and unstable foliations respectively. We denote as
$E^{cs}_X$, $E^{cu}_X$, $E^{ss}_X$ and $E^{uu}_X$ to the tangent
bundles of these foliations.
\end{rema}

The above construction immediately generalizes to examples arising
from $DA$ maps with $n$ sources instead of $2$, where $n \geq 1.$
We call these flows \textit{Franks-Williams type Anosov flows.}
The arguments in the next sections are very general and apply to
all Franks-Williams type Anosov flows.

\section{A perturbation on a model space}\label{s.perturbation}

In this section we shall perform a perturbation in a certain model
space.

Consider the torus $T_1$ and $X_1(T_1)$ where $X_1$ is the time
one map of the flow of $X$. Then $T_1$ and $X_1(T_1)$ bound a
manifold diffeomorphic to $[0,1]\times \TT^2$, which is a
fundamental domain of $X_1$. We shall use the convention that sets
of the form $[a,b] \times \{p\}$ are \emph{horizontal} and those
of the form $\{t\} \times \TT^2$ are \emph{vertical}.

The projection of the vector-field $X$ on this coordinates is
$\frac\partial{\partial t}$.

We denote by $\cF^{ss}$ and $\cF^{uu}$ the $1$-dimensional
foliations $(\{t\}\times F^s_1)_{t\in [0,1]}$ and $(\{t\}\times
F^u_1)_{t\in  [0,1]}$ and by $\cF^{cs}$ and $\cF^{cu}$ the
$2$-dimensional foliations $ [0,1]\times F^s_1$ and $ [0,1]\times
F^u_1$.

\begin{lemm}\label{l.transverse}Let $\cG\colon [0,1]\times \TT^2\to [0,1] \times \TT^2$
be the diffeomorphism defined as $(t,x,y)\mapsto (t,x,y+\rho(t))$.
Where $\rho: [0,1] \to [0,1]$ is a monotone smooth function such
that it is identically zero in a neighborhood of $0$ and
identically $1$ in a neighborhood of $1$.

Then $\cG(\cF^{cu})$ is transverse to $\cF^{ss}$ and
$\cG(\cF^{uu})$ is transverse to $\cF^{cs}$.

\end{lemm}

\begin{proof} This is a direct consequence of Lemma
\ref{l.coordinates}, Corollary \ref{c.transverse} and  the fact
that $\cG$ makes translations only in the $y$-direction of the
coordinates given by that lemma.
\end{proof}

\section{The strong stable and unstable foliations on a fundamental domain of $X_N$ for large $N$}\label{s.N}

For any $N>0$ we consider the fundamental domain $\cU_N$ of the
diffeomorphism $X_N$ (time $N$ of the flow of $X$) restricted to
$M \setminus (A_X \cup R_X)$ bounded by $T_1$ and $X_N(T_1)$.

This fundamental domain $\cU_N$ is canonically  identified with
$[0,N] \times T_1$: the projection of $T_1$ is the identity map on
$\{0\}\times T_1$ and the projection of $X$ is
$\frac\partial{\partial t}$.

The intersection of the weak stable and weak unstable
$2$-foliations $W^{cs}_X$, $W^{cu}_X$  as well as the
($1$-dimensional) strong stable and strong unstable foliations
$W^{ss}_X$ and $W^{uu}_X$ with $\cU_N$; will be denoted by
$W^{cs}_N$, $W^{cu}_N$ $W^{ss}_N$, $W^{uu}_N$.

\begin{lemm}\label{l.expressionindep} The expression of the tangent space of  $W^{cs}_N$, $W^{cu}_N$
$W^{ss}_N$, $W^{uu}_N$ at a point $(t,x,y)\in [0,N] \times T_1$
only depends on $(x,y)\in T_1$ : it
 depends neither of $t\in [0,N]$ nor on $N$.
\end{lemm}

\begin{proof}
This follows from the fact that the bundles are invariant under
the flow and the vector-field $X$ is $\frac{\partial}{\partial t}$
in this coordinates.
\end{proof}

We denote by $H_N\colon \cU_N \to [0,1]\times \TT^2$ the
diffeomorphisms defined by $H_N(t,p)=(\frac tN,
\theta_1^{-1}(p))$.

We denote $\cF^{cs}_N=H_N(W^{cs}_N)$, $\cF^{cu}_N=H_N(W^{cu}_N)$,
$\cF^{ss}_N=H_N(W^{ss}_N)$, $\cF^{uu}_N=H_N(W^{uu}_N)$.

\begin{lemm}\label{l.limit} For every $N>0$, $\cF^{cs}_N=\cF^{cs}$ and $\cF^{cu}_N=\cF^{cu}$, where
$\cF^{cs}$ and $\cF^{cu}$ are the $2$-dimensional foliation on $
[0,1]\times \TT^2$ defined in Section~\ref{s.perturbation}.

The tangent bundle to $\cF^{ss}_N$ and to $\cF^{uu}_N$ converges
in the $C^0$ topology to the tangent bundle to the foliations
$\cF^{ss}$ and $\cF^{uu}$, defined in
Section~\ref{s.perturbation}, as $N$ goes to $+\infty$.

\end{lemm}

\begin{proof}
The first assertion is a direct consequence of the fact that
$\cF^{cs}_N$ and $\cF^{cu}_N$ are saturated by the orbits of the
flow and Lemma \ref{l.expressionindep}.

 The second assertion follows from the
fact that the diffeomorphism is independent of $N$ in the first
coordinate and compresses the $t$ coordinate by $\frac 1N$.

\end{proof}

As a direct consequence of Lemmas~\ref{l.transverse} and
~\ref{l.limit} we obtain

\begin{coro}\label{c.transverse} Let $\cG$ be the diffeomorphism defined in Lemma~\ref{l.transverse}.
Then , there is $N_0>0$ so that for any $N\geq N_0$  one  has:

$\cG(\cF^{cu}_N)$ is transverse to $\cF^{ss}_N$ and
$\cG(\cF^{uu}_N)$ is transverse to $\cF^{cs}_N$.

\end{coro}

\section{Establishing partial hyperbolicity}\label{s.Criteria}

We just recall a classical criterion for partial hyperbolicity. As
before, we remain in dimension 3 for simplicity.

Let $g$ be a diffeomorphism on a compact $3$-dimensional manifold
$M$. Assume that there is a codimension $1$ submanifold cutting
$M$ in two compact manifolds with boundary $M^+$ and $M^-$ so that
$M^+$ is an \emph{attracting region} for $f$ (i.e. $f(M^+)\subset
\mathrm{Int}(M^+)$) and $M^-$ is a \emph{repelling region} (i.e.
attracting region for $g^{-1})$.

Assume that the maximal invariant set $A$ in $M^+$ and the maximal
invariant set $R$ in $M^-$ admit a (absolute) partially hyperbolic
splitting $E^{ss}\oplus E^c\oplus E^{uu}$.

Then $E^{ss}$ and $E^{ss}\oplus E^c$ admit a unique invariant
extension $E^{ss}$ and $E^{cs}$ on $M\setminus R$ and
symmetrically $E^{uu}$ and $E^{c}\oplus E^{uu}$ admit a unique
invariant extension $E^{uu}$ and $E^{cu}$ on $M\setminus A$.

\begin{rema}\label{r.uniqueextensionAnosovflows}
Assume that $g$ coincides with an Anosov flow $Z$ in a
neighborhood $U$ of $A$. Then, the bundles $E^{cs}$ and $E^{ss}$
in $M\setminus R$ coincide exactly with the tangent spaces of the
foliation $f^{-n}(W^{cs}_Z\cap U)$ and $f^{-n}(W^{ss}_Z \cap U)$.
A symmetric property holds for $E^{cu}$ and $E^{uu}$. Notice that
the center bundle cannot be a priori extended to these sets.
\end{rema}

A compact set $\cU$ will be called  a \emph{fundamental domain} of
$M \setminus (A \cup R)$ if every orbit of $g$ restricted to $M
\setminus (A \cup R)$ intersects $\cU$ in at least one point.

\begin{theo}\label{t.criterion} Let $g: M
\to M$ be a $C^1$-diffeomorphism. Assume that:
\begin{itemize}
\item there exists a codimension one submanifold cutting $M$ into
an attracting and a repelling regions with maximal invariant sets
$A$ and $R$ which are (absolutely) partially hyperbolic;

\item there exists a compact fundamental domain $\cU$ of
$M\setminus (A\cup R)$ such that if $E^{ss}_A$,$E^{cs}_A$ denote
the extensions of the bundles $E^{ss}$ and $E^{ss}\oplus E^c$ of
$A$ to $M \setminus R$ and $E^{uu}_R$, $E^{cu}_R$ denote the
extensions of $E^{uu}$,$E^{c}\oplus E^{uu}$ of $R$ to $M\setminus
A$ then $E^{cs}_A$ is transverse to $E^{uu}_R$ and $E^{cu}_R$ is
transverse to $E^{ss}_A$ at each point of $\cU$.
\end{itemize}
Then $g$ is (absolutely) partially hyperbolic on $M$.
\end{theo}

\begin{proof} We have a well defined splitting $E^{ss} \oplus E^c \oplus E^{uu}$ above $A$ and $R$
and we can define the bundles $E^{ss}$ and $E^{uu}$ everywhere as
$E^{ss}:= E^{ss}_A$ and $E^{uu}:=E^{uu}_R$ in $M\setminus (A \cup
R)$.

The transversality conditions we have assumed on $\cU$ allows us
to define the $E^c$ bundle in $M \setminus (A \cup R)$ as the
intersection between $E^{cs}_A$ and $E^{cu}_R$. These intersect in
a one-dimensional subbundle thanks to our transversality
assumptions.

Now, let us show that the splitting we have defined is
(absolutely) partially hyperbolic. To see this, it is enough to
show that the decomposition is continuous. Indeed, if this is the
case, one can use the fact that given a neighborhood $U$ of $A
\cup R$ there exists $N>0$ such that every point outside $U$
verifies that every iterate larger than $N$ belongs to $U$.
Together with continuity of the bundles and the partially
hyperbolicity along $A \cup R$ this allows to show that if
$g|_{A\cup R}$ is (absolutely) partially hyperbolic, then it must
be (absolutely) partially hyperbolic globally.

To show continuity of the bundles, notice first that since for a
point $x\in M \setminus (A\cup R)$ the bundle $E^{uu}_R$ is
transverse to $E^{cs}_A$ one has that as one iterates forward the
point $g^n(x)$ approaches $A$ while the bundle $D_x g^n(E^{uu}_R)$
must approach $E^{uu}$ by the transversality and domination. The
same argument shows that $E^{ss}$ is also continuous. The fact
that $E^{c}$ glues well with $E^c$ along $A$ follows from the fact
that $E^{cs}_A$ is invariant and $E^{c}$ is transverse to
$E^{ss}_A$ in $E^{cs}_A$. The symmetric argument gives continuity
of $E^{c}$ as one approaches $R$ and this concludes the proof.

\end{proof}

\section{The example: perturbation of the time $N$ of the
flow}\label{s.example}

In this section we construct the example announced in Theorem
\ref{t.main} and prove it is (absolute) partially hyperbolic.

We fix $N\geq N_0$ as in Corollary~\ref{c.transverse}.

Consider  $M\setminus (A_X\cup R_X)$.  Let $V_1$ be the
$X$-invariant open subset of $M$ consisting in the point whose
orbit crosses $T_1$.

Consider the diffeomorphism $\cG$ on $[0,1] \times \TT^2$ defined
in Lemma \ref{l.transverse}. We consider $G: M \to M$ defined as
the identity outside $\cU_N$ and $G= H_N^{-1} \circ \cG \circ H_N$
in $\cU_N$. Then $G$ is a smooth diffeomorphism from how $\cG$ was
defined.

Define the diffeomorphism $f: M \to M$:

$$ f= G \circ X_N .$$

We will prove that the diffeomorphism $f$ defined above satisfies
all the conclusion of Theorem~\ref{t.main}. We begin by first
demonstrating that $f$ is partially hyperbolic.

\begin{theo}\label{p.partiallyhyp} The diffeomorphism $f$ is absolutely partially hyperbolic.
\end{theo}

\begin{proof}
First notice that the chain recurrent set of $f$ coincides with
the non-wandering set of $X_N$ and thus of $X$, that is $A_X\cup
R_X$. Furthermore the dynamics of $f$ coincides with the one of
$X_N$ in a neighborhood of $A_X\cup R_X$. In particular, $f$ is
absolute partially hyperbolic in restriction to $A_X\cup R_X$.

According to Theorem~\ref{t.criterion} we must check that the
extension $E^{ss}_A$ and $E^{cs}_A$ of the bundles $E^{ss}_X$,
$E^{cs}_X$ in $A_X$ and the extensions $E^{cu}_R$, $E^{uu}_R$ of
$E^{cu}_X$ and $E^{uu}_X$ on $R_X$ satisfy the transversality
conditions between $E^{cs}_A$ and $E^{uu}_R$ and $E^{ss}_A$ and
$E^{cs}_R$ in a compact fundamental domain $\cU$.

Recall that $E^{ss}_A$, $E^{cs}_A$ in a neighborhood of $A_X$ are
the tangent bundles to the strong stable and center stable
foliations of $A_X$ and $E^{cu}_R$, $E^{uu}_R$ coincide with the
tangent bundles to the strong unstable and center unstable
foliations of $R_X$ (see Remark \ref{r.uniqueextensionAnosovflows}
and notice that $f$ coincides with an Anosov flow in neighborhoods
of $A_X$ and $R_X$).

Notice that $f$ admits a fundamental domain having two connected
components, one (denoted as $\De_{1,N}$) bounded by $T_1$ and
$X_N(T_1)$ and the other bounded by $T_2$ and $X_N(T_2)$ (denoted
as $\De_{2,N}$). Therefore $M\setminus (A_X\cup R_X)$ has two
connected components $V_1$ and $V_2$ which are the sets of points
whose orbits pass through one or the other fundamental domains.

On $V_2$ the diffeomorphism $f$ coincides with $X_N$ and the
bundles coincide with the invariant bundles of $X$, so that we get
the transversality conditions for free.

Thus we just have to show the transversality in $\De_{1,N}=
\cU_N$.

 Notice that $f$ coincides with $X_N$ on $$\bigcup_{t\geq 0} X_t(T_1).$$
 As $\bigcup_{t\geq 0} X_t(T_1)$ is positively invariant and the orbits
 tends to $A_X$ the stable and center stable foliations
 of $f$ coincide on that set with those of $X$.

 In particular, $E^{cs}_A$ and $E^{ss}_A$ in $\cU_N$ coincide with $E^{cs}_X$ and $E^{ss}_X$ respectively.

Notice also that $f$ coincides with $X_N$ on $$\bigcup_{t\leq -N}
X_t(T_1).$$
 As $\bigcup_{t\leq -N} X_t(T_1)$ is negatively invariant and the negative orbits
 tends to $R_X$ the unstable and center unstable foliations
 of $f$ coincide on that set with those of $X$.

 In particular, we obtain that $E^{cu}_R= G_\ast (E^{cu}_X)$ and $E^{uu}_R= G_\ast(E^{uu}_X)$ in $\cU_N$ which satisfy the transversality
 conditions thanks to Corollary \ref{c.transverse}.

 \end{proof}

The bundles of $f$ will be denoted as $E^{ss}_f$, $E^c_f$ and
$E^{uu}_f$. As usual, we denote $E^{cs}_f= E^{ss}_f \oplus E^c_f$
and $E^{cu}_f = E^{c}_f \oplus E^{uu}_f$.

\section{Dynamical coherence}\label{s.coherence}
We are now ready to prove that the diffeomorphism $f$ is robustly dynamically
coherent and that it cannot be leaf conjugate to the time one map
of an Anosov flow. Furthermore, the same result holds for any iterate and
any finite lift of $f.$

\begin{lemm}\label{l.Lyapunov} There exists $K>1$ so that for any unit vector $v\in E^c_f$ and any $n\in\ZZ$ one has
$$\frac1 K< \|Df^n(v)\|< K.$$
\end{lemm}
\begin{proof}Notice that given any neighborhoods $U_R$ and $U_A$ of $R_X$ and $A_X$, there is $n_0>0$ so that
$f^{n_0}(M\setminus U_R)$ is contained in $U_A$.

We choose $U_A$ as being a positively $X_t$-invariant neighborhood
of $A_X$ on which $f$ coincides with $X_N$. In $U_A$ the bundles
$E^{ss}_f$ and $E^{cs}_f$ coincide with  $E^{ss}_X$  and
$E^{cs}_X$. As $E^c_f$ is transverse to $E^{ss}_f$ on the compact
manifold $M$, we get that any unit vector of $E^c_f$ in $U_A$ has a
component in $\RR X$ uniformly bounded from below and above and a
component in $E^{ss}_X$  uniformly bounded (from above). Therefore,
the positive iterates of $v$ by $Df$ are uniformly bounded from
below and above.

The same phenomena occur in a neighborhood $U_R$ of $R_X$ when we consider
backward iterates. Now the result is established by using the invariance of
$E^c_f$ and the fact that any orbit spends at most $n_0$ iterates
outside $U_R\cup U_A$.
\end{proof}

The inequalities of Lemma~\ref{l.Lyapunov} force the curves tangent to the
center direction  to possess a very useful dynamical property; this is the content of Corollary~\ref{c.Lyapunov}. First, let us recall a definition.

\begin{defi}
Consider  a diffeomorphism $g$  with an invariant bundle $E
\subset TM$, one says that $g$ is \emph{Lyapunov stable} in the
direction of $E$ if given  any $\varepsilon>0$ there is $\delta>0$
so that any path $\gamma$ tangent to $E$ of length smaller than
$\delta$ verifies that the forward iterates $g^n(\gamma)$ have
length smaller than $\varepsilon$.
\end{defi}

As a direct consequence of Lemma \ref{l.Lyapunov} one gets:

\begin{coro}\label{c.Lyapunov}
The diffeomorphism $f$ is Lyapunov stable in the direction
$E^{cs}_f$. Symmetrically, $f^{-1}$ is Lyapunov stable in the
direction $E^{cu}_f$.
\end{coro}

\begin{proof}
Just notice that for any unit vector $v$ tangent to $E^{cs}_f$ the
forward iterates $Df^n(v)$ have uniformly bounded norm.
\end{proof}
This permits us to demonstrate

\begin{theo}\label{p.dynamicalcoherence} The diffeomorphism $f$ is robustly dynamically
coherent. Moreover, the bundle $E^c_f$ is uniquely integrable.
\end{theo}

\begin{proof} We refer the reader to \cite[Section 7]{HPS} or
\cite[Section 7]{HHU3} for precise definitions of some notions
which will appear in this proof (which are classical in the theory
of partially hyperbolic systems). Related arguments appear in
\cite{BBI}.

It is shown in \cite[Theorem 7.5]{HHU3} (see also \cite[Theorem
7.5]{HPS}) that $E^{cs}_f$ is tangent to a unique foliation
provided $f$ is Lyapunov stable with respect to $E^{cs}_f$.
Therefore, by Corollary \ref{c.Lyapunov} the bundle $E^{cs}_f$ is
tangent to a unique foliation tangent to $E^{cs}_f$. Moreover, it
is also established that under this assumptions, the unique foliation
$W^{cs}_f$ must be \emph{plaque-expansive} in the sense of
\cite[Section 7]{HPS}.

Applying the same result for $f^{-1}$ and $E^{cu}_f$ we deduce
dynamical coherence, and using \cite[Theorem 7.1]{HPS}, we may conclude
that the center-stable and center-unstable foliations are
\emph{structurally stable} (in particular, they exist for small
$C^1$ perturbations of $f$). This concludes the proof of robust
dynamical coherence.

Finally, unique integrability of $E^{c}_f$ follows by classical
arguments using the fact that curves tangent to $E^{c}_f$ are
Lyapunov stable for $f$ and $f^{-1}$ because of Lemma
\ref{l.Lyapunov} (see also \cite[Corollary 7.6]{HHU3}).

\end{proof}

Now we are ready to prove that the example cannot be leaf conjugate to the
time-one map of an Anosov flow---this is the result advertised
in Corollary \ref{c.conjecture}.
However, first we must establish an important property about the leaves of the
center-foliation $W^c_f$ of $f$. Recall from the previous section
that $\cU_N= \bigcup_{0 \leq t \leq N} X_t(T_1)$.

\begin{lemm}\label{l.centerfoliation}
The connected components of $W^c_f \cap \cU_N$ are arcs which join
$T_1$ with $X_N(T_1)$ with uniformly bounded length.
\end{lemm}

\begin{proof} This follows immediately  from the fact that the perturbation
preserves the $\TT^2 \times \{t\}$ coordinates (in the coordinates
given by $H_N$) and the original center direction was positively
transverse to those fibers. Since $\TT^2 \times [0,1]$ is compact,
one obtains that these arcs are of bounded length and join both
boundaries of $\cU_N$.
\end{proof}

We can now show:

\begin{theo}\label{p.noleafconj} There are center leaves which
are not fixed for no iterate of $f$. Consequently, there is no
finite lift or finite iterate of $f$ which is leaf conjugate to
the time-one map of an Anosov flow.
\end{theo}

\begin{proof} Using  unique integrability, one knows that the
$f$-invariant foliation $W^c_{f}$ tangent to $E^c_f$ is obtained
by intersecting the preimages of $W^{cs}_X$ with the forward
images of $W^{cu}_X$, where $W^{cs}_X$ and $W^{cu}_X$ are defined
in neighborhoods of $A_X$ and $R_X$ respectively.

In particular, we get that restricted to $\cU_N$ which is a
fundamental domain for $f$ in the complement of $A_X \cup R_X$, we
have that $W^c_{f}$ consists of $W^{cs}_X \cap G(W^{cu}_X)$ which
is an arc joining $T_1$ and $X_N(T_1)$ as proved in Lemma
\ref{l.centerfoliation}.

Notice that $f$ coincides with $X_N$ in a neighborhood of $T_1$
(and $G=id$ in a neighborhood of $T_1$) so that $W^c_{f}$ consists
of horizontal lines (i.e. of the form $[0,\varepsilon) \times
\{p\} $ or $(1-\varepsilon,1] \times \{p\}$ in the coordinates
given by $H_N$) in neighborhoods of $T_1$ and $f(T_1)$. Moreover,
one has that the image of the center line which is of the form
$[0,\varepsilon)\times \{p_0\}$ in a neighborhood of $T_1$ is sent
by $f$ to the arc of $W^c_f$ which is of the form
$(1-\varepsilon,1]\times \{p_0\}$ (again because $G=id$ in a
neighborhood of $T_1$ and $X_N(T_1)$).

On the other hand, those arcs cannot be joined inside $W^{cs}_X
\cap G(W^{cu}_X)$ if they are not the leaves corresponding to
circles in $T_1$, and  so, it follows that the center leaves cannot
be fixed by $f$. The same argument implies that this is not
possible for $f^k$ for any $k\geq 0$ since $f$ coincides with
$X_N$ once it leaves $\cU_N$.
\end{proof}

\begin{rema} This in stark contrast with the results of
\cite{BW} where it is shown that when $f$ is transitive, if
certain center leaves are fixed, then all center leaves must be.
It is also important to note that in this example the leaves of
both the center-stable and center-unstable foliations are fixed by
$f$ but the connected components of their intersections---the
center leaves---are not.
\end{rema}

Since the $3$-manifold $M$ admits an Anosov flow, $\pi_1(M)$ has
exponential growth, and so, the manifold $M$ does not support a
partially hyperbolic diffeomorphisms isotopic to an Anosov
diffeomorphism or a skew product.  This is because Anosov
diffeomorphisms only exist on $\TT^3$ in dimension 3 (see for
example \cite{HP} and references therein) and skew products are
only defined (in \cite{BW}) on circle bundles over $\TT^2$---both
these types of 3-manifolds are associated with polynomial growth
in their fundamental groups. These growth properties are immune to
taking finite covers, and so, the same holds true for any finite
cover of $M.$ Therefore, if a finite lift or iterate of $f$ were
leaf conjugate to one of the three models, it would have to be the
time-one map of an Anosov flow.  In this case, there would exist
an iterate of $f$ on a finite cover that would fix every center
leaf. We have shown that this is not the case, and thus, Theorem
\ref{p.noleafconj} implies Corollary \ref{c.conjecture}.

\section{Non-trivial isotopy class}\label{s.isotopy}
In this section we will complete the proof of  Theorem
\ref{t.main} by showing that no iterate of $f$ is isotopic to the
identity. It is at this point only that we shall use the fact that
the original $DA$-diffeomorphism has at least two sources since this provides a curve intersecting the torus
where the modification is made which is homologically non trivial. So we restrict ourselves
to the case where there are exactly two sources, but it should be
clear that all we have done until now works with any number of
such sources.

\begin{theo}\label{p.notisotopic} For every $k \neq 0$, $f^k$ is not
isotopic to the identity. More precisely, the action of $f^k$ on
homology is non-trivial.
\end{theo}

\begin{proof}
It suffices to show that the action on homology is not trivial.

To establish this, first note that for the suspension manifold
$M_0$ the periodic orbits where we did the $DA$-construction are
homologically non-trivial. This implies that after removing the
solid torus, the circles in the same direction in the boundary are
still homologically non-trivial. The gluing we have performed
preserves this homology class, so that it remains homologically
non-trivial after the gluing too \footnote{A way to see this is by
constructing a closed 1-form in each piece which integrates one in
the desired circle. These 1-forms glue well to give a closed
1-form integrating one in the same curve.}.

Notice moreover that there is a representative $\gamma_1$ of this
homology class which does not intersect $T_1$ simply by making a
small homotopy of this loop which makes it disjoint from $T_1$.

Consider a closed curve $\gamma_2$ intersecting $T_1$ only once
(it enters by $T_1$ and then comes back by $T_2$) we know that
$\gamma_2$ is not homologous to $\gamma_1$. This can be shown by
considering a small tubular neighborhood $U$ of $T_1$ and a closed
$1$-form which is strictly positive\footnote{If $U\sim T_1 \times
[-1,1]$ and we call $\theta$ to the variable on $[-1,1]$ it
suffices to choose $f(\theta)d\theta$ with $f$ smooth, positive on
$(-1,1)$ and vanishing at $\pm 1$.} in $U$ and vanishes outside
$U$. It is clear that such a $1$-form has a non-vanishing integral
along $\gamma_2$ and a vanishing one along $\gamma_1$ proving the
desired claim.

From how $G$ was chosen it is clear that $G^k_\ast ([\gamma_2]) =
[\gamma_2 ] + k[\gamma_1]$. Since $X_N$ is isotopic to the
identity, we also have that $(G \circ X_N)_\ast = G_\ast$.

This implies that for every $k\neq 0$, the map $(f^k)_\ast= (G
\circ X_N)_\ast^k = G_\ast^k$ is different from the identity.

\end{proof}

\section*{Acknowledgements}
We would like to thank the hospitality of IMERL-CMAT (Montevideo,
Uruguay) and the Institut de Math\'ematiques de Burgogne (Dijon,
France) for hosting us in several stages of this project.

\vspace{1.5cm}

\begin{itemize}
\item[]  \emph{Christian Bonatti}
\begin{itemize}
\item[] Institut de Math. de Bourgogne CNRS - UMR 5584 \item[]
Universit\'e de Bourgogne. \item[] Dijon 21004, France
\end{itemize}
\item[] \emph{Kamlesh Parwani}
\begin{itemize}
\item[] Department of Mathematics \item[]Eastern Illinois
University \item[]Charleston, IL 61920, USA.
\end{itemize}
\item[] \emph{Rafael Potrie}
\begin{itemize}
\item[]Centro de Matem\'atica, Facultad de Ciencias \item[]
Unviversidad de la Republica \item[]11400, Montevideo, Uruguay
\end{itemize}
\end{itemize}

\end{document}